\documentclass[a4paper,10pt,reqno]{amsart}

\usepackage[textsize=footnotesize,color=yellow, bordercolor=white]{todonotes}

\usepackage[applemac]{inputenc}  %TODO: SWITCH TO UTF-8
\usepackage[T1]{fontenc}
\usepackage[swedish,english]{babel}
\usepackage{csquotes}

\usepackage{amssymb}
\usepackage{amsmath,amsthm}

\usepackage{datetime}

\usepackage[shortlabels]{enumitem}

\usepackage[pdfauthor={},
            pdftitle={A closed maximal eventually different family},
            pdfproducer={Latex with hyperref}]{hyperref}

\theoremstyle{plain}
\newtheorem{thm}{Theorem}[section]
\newtheorem{lem}[thm]{Lemma}

\newtheorem{cor}[thm]{Corollary}
\newtheorem{claim}[thm]{Claim}

\newtheorem*{question*}{Question}

\theoremstyle{definition}
\newtheorem{dfn}[thm]{Definition}

\newtheorem{fct}[thm]{Fact}

\newtheorem{rem}[thm]{Remark}

\theoremstyle{remark}

{\end{minipage}\end{equation}}

\newenvironment{eqpar*}{\begin{equation*}\begin{minipage}{0.8\columnwidth}}%
{\end{minipage}\end{equation*}}

\providecommand{\nat}{\mathbb{N}}

\DeclareMathOperator{\lh}{lh}

\providecommand{\res}{\mathbin{\upharpoonright} }
\providecommand{\conc}{ \mathbin{{}^\frown}}

\providecommand{\Hhier}{\mathbf{H}}

\providecommand{\setdef}{\;|\;}

\newcommand{\secret}[1]{}

\DeclareMathOperator{\powerset}{\mathcal{P}}
\DeclareMathOperator{\codefin}{\#_2}
\newcommand{\codefinF}[1][F]{\operatorname{\#_2^{#1}}}

\DeclareMathOperator{\codesimple}{\#_1}

\DeclareMathOperator{\ed}{e}
\DeclareMathOperator{\edh}{\dot e}
\DeclareMathOperator{\edg}{\ddot e}

\newcommand{\order}[1][{}]{\mathbin{\prec^{#1}}}

\DeclareMathOperator{\nuhof}{B}

\newcommand{\I}[2][{}]{\operatorname{I}^{#1}_{#2}}

\newcommand{\pspace}[1]{\mathcal N_{ #1 }}

\newcommand{\fin}[1]{#1^*}

\title{Compactness of maximal eventually different families}

\author{David Schrittesser}
\address{Department of Mathematical Sciences, University of Copenhagen, Universitetsparken 5, 2100 Copenhagen, Denmark}
\email{david.s@math.ku.dk}

\subjclass[2010]{03E15, 03E05}

\keywords{effectively, compact, closed, maximal, eventually, different, family}

\begin{document}

\newcommand{\linspan}[1]{\hull( #1 )}

\begin{abstract}
We show that there is an effectively closed maximal eventually different family of functions in spaces of the form $\prod_n F(n)$ for $F\colon \nat \to \nat\cup\{\nat\}$ (e.g., Baire space) and give an exact criterion for when there exists an effectively compact such family. 
The proof generalizes and simplifies those in \cite{medf-borel} and \cite{medf-closed}.
\end{abstract}

\maketitle

\section{Introduction}

\paragraph{A} In \cite{medf-borel} Horowitz and Shelah construct in ZF a maximal eventually different family (short: \emph{medf}) which is $\Delta^1_1$, i.e., effectively Borel. This was a surprise since, e.g., infinite so-called \emph{mad} families cannot even be $\mathbf{\Sigma}^1_1$ (i.e., \emph{analytic}; see \cite{mathias.1967,tornquist.2015}).
In a more recent, related result \cite{mcg-borel} they obtain a $\Delta^1_1$ \emph{maximal cofinitary group}.

\medskip

The present paper answers a question of Asger Törnquist \cite{personal}: 
Given $F\colon \nat\to\nat$ such that $\liminf_{n\to\infty} F(n) = \infty$ does  there exist a Borel or even a compact \emph{medf} in the space 
$\{g \in {}^\nat\nat \setdef (\forall n\in \nat)\; g(n) < F(n) \}$? 
%$\pspace{ F }$? Here  the space $\pspace{ F }$ is defined to be the closed subspace 
%$
%%\[
%\pspace{ F } = \{g \in {}^\nat\nat \setdef (\forall n\in \nat)\; g(n) < F(n) \}
%%\]
%$ of Baire space $\nat^\nat$.
As well as answering this question, we  construct  a \emph{medf} (which is $\Pi^0_1$, i.e., \emph{effectively closed}) in Baire space in an even more elementary way than in \cite{medf-borel} or \cite{medf-closed}.

\medskip

To make the question entirely precise, we give the definition of (maximal) eventually different families a broader context:
\begin{dfn}
Any two functions $g_0, g_1$ with domain $\nat$ are called \emph{eventually different} if and only if
$\{n\in \nat \setdef g_0(n) = g_1(n) \}$ is finite.

Given a function $F\colon \nat\to\nat\cup\{\infty\}$, let $\pspace{ F } = \{g \in {}^\nat\nat \setdef (\forall n\in \nat)\; g(n) < F(n) \}$  (with the product topology, and $\{ k\in\nat \setdef k< F(n)\}$ discrete, as usual).
%and equip $\pspace{ F }$, i.e., $\prod_{n\in \nat} F(n)$ with the product topology (where as always $F(n)$ is identified with $\{ k \in \nat \setdef k < F(n) \}$).
A set $\mathcal E$ is an \emph{eventually different family in $\pspace{ F }$} if and only if $\mathcal E \subseteq \pspace{ F }$ and  any two distinct $g_0,g_1\in\mathcal E$ are eventually different; such a family is called \emph{maximal} (or short: a \emph{medf}) if and only if it is maximal among such families under inclusion. 
\end{dfn}
We now state our main result, followed by a rather straightforward corollary.
\begin{thm}\label{t.main} 
Suppose $F \colon \nat \to (\nat \setminus\{0\})\cup\{\infty\}$. 
If $\lim_{n\to\infty} F(n) = \infty$ 
there is a perfect $\Pi^0_1(F)$ (effectively-in-$F$ closed) maximal eventually different family in $\pspace{ F }$. 
\end{thm}

\begin{cor}\label{t.general} 
Suppose $F \colon \nat \to (\nat \setminus\{0\})\cup\{\infty\}$. There is a compact $\Pi^0_1(F)$ \emph{medf}  in $\pspace{ F }$ if and only if $F(n) < \infty$ for infinitely many $n\in \nat$. Moreover, exactly one of the following holds:
\begin{enumerate} 
\item 
Every \emph{medf} is finite and there is a finite \emph{medf} consisting of constant functions (namely, when $\liminf_{n\to\infty} F(n) < \infty$);
\item There is a perfect $\Pi^0_1(F)$ \emph{medf} but no countable \emph{medf}.
\end{enumerate}
\end{cor}

We ask $(\forall n\in \nat) \; F(n) > 0$ to preclude the trivial case of the empty space. 
As $\liminf_{n\to\infty} F(n) < \infty$ means there is $m$ such that $\{n \in \nat\setdef F(n) = m \}$ is infinite, 
the set $\{ c_k \setdef k < m \}$ where $c_k$ is the constant function with value $k$ constitutes a \emph{medf} in this case. 
So the question posed by Törnquist is only interesting if $\lim_{n\to\infty} F(n) = \infty$ holds.

\medskip

\paragraph{B} 
We fix some notation and terminology (see also \cite{kechris}). 
For any set $X$, $|X|$ denotes its cardinality.
As always in set theory, we identify $n$ with $\{ k \in \nat \setdef k < n \}$, $\nat$ with $\omega$ and $\nat\cup\{\nat\}$ with $\omega+1$,
$\nat$ is equipped with the discrete and $\omega+1$ with the order topology.
We write ${}^A B$ to mean the set of functions from $A$ to $B$ and ${}^{<\nat} B$ means the set of finite sequences from $B$.
If $s$ is a sequence, $\lh(s)$ denotes its length.
For a set $A$ and a function $F$ on $A$, by $\prod_{x \in A} F(x)$ we always mean the set of functions $f$ with domain $A$ such that $(\forall a \in A) \; f(a) \in F(a)$ (not a product of numbers).
Both ${}^\nat \nat$ and $\prod_{x \in A} F(x)$ for $F$ as in Theorem~\ref{t.main} naturally carry the product topology.

We write $f_0 \mathbin{=^\infty} f_1$ to mean that $f_0$ and $f_1$ are \emph{not} eventually different (they are infinitely equal). Two sets $A, B \subseteq \nat$ are called \emph{almost disjoint} if and only if $A \cap B$ is finite, and
an \emph{almost disjoint family} is a set $\mathcal A \subseteq \powerset(\nat)$ any two elements of which are almost disjoint. 
We write $A \subseteq^* B$ to mean $A$ is \emph{almost contained in } $B$, i.e., $A\setminus B$ is finite.

We naturally take `\dots is recursive in \dots' to apply to subsets of $\Hhier(\omega)$, the set of hereditarily finite sets.
Any function $F \colon \nat \to \omega+1$ is for this purpose identified with a subset of  $\Hhier(\omega)$ by replacing the value $\omega$ with some fixed element $\infty$ of $\Hhier(\omega)\setminus\omega$. Consult \cite{moschovakis,mansfield} for more on the (effective) Borel and projective hierarchies, i.e., on $\Pi^0_1$, $\Pi^0_1(F)$,  $\Delta^1_1$, \dots{} sets.

Clearly all of the results in this paper can be derived in ZF (in fact in a not so strong subsystem of second order arithmetic).

\medskip

\paragraph{C}
This note is organized as follows.
In Section~\ref{s.main.lemma} we prove a special case of Theorem~\ref{t.main} assuming that $F$ grows quickly enough so that we can code initial segments of functions without running out of space, in Lemma~\ref{l.growth}. 
The construction we give also applies to Baire space.
In Section~\ref{s.main} we prove Theorem~\ref{t.main} in full generality, quoting the proof of Lemma~\ref{l.growth}.
In Section~\ref{s.cor} we prove two simple facts which together with Theorem~\ref{t.main} imply Corollary~\ref{t.general}.
We close with some open questions in Section~\ref{s.q}.

\medskip

{\it Acknowledgements: The author gratefully acknowledges
 the generous support from the DNRF Niels Bohr Professorship of Lars Hesselholt. 
}

 \section{The main lemma}\label{s.main.lemma}

We start by proving a variant of Theorem~\ref{t.main} (that there is a $\Pi^0_1$ \emph{medf} in $\mathcal N_F$) assuming $F$ satisfies a growth condition.
The proof will at the same time make more elementary and generalize the construction of Horowitz and Shelah \cite{medf-borel} and the present author's version \cite{medf-closed}.
\begin{lem}\label{l.growth}
Supposing $F\colon\nat \to \omega + 1$ is such that  for all $n\in \nat$
\begin{equation}\label{e.growth}
  \sum_{l\leq n}\Big(  \Big\lvert \prod_{k < l}  F(k) \Big\rvert \cdot 2^{l} \Big) \leq F(n),
\end{equation}
there is a perfect $\Pi^0_1(F)$ maximal eventually different family in $\pspace{ F }$.
\end{lem}
Clearly, if $F(n) < \omega$ for every $n\in \nat$ then this family will be compact. 
Note that $\prod \emptyset$ gives $\{\emptyset \}$, i.e., the set containing the empty function, whence \eqref{e.growth} implies $F(0) > 0$ and in fact $(\forall n\in \nat) \; F(n) > 0$.

\medskip

Before we begin with the proof, we introduce the main ingredients of the construction and define our \emph{medf} $\mathcal E$. 
The definitions and proofs which follow are a further streamlined version of those in \cite{medf-closed}, where the reader will find more explanations.
\begin{dfn}\label{d.tools}
Suppose $F$ satisfies \eqref{e.growth}, $g \in \pspace{F}$ and $c \in {}^\nat 2$.
\begin{enumerate}[label=(\Alph*),ref=\Alph*]

\item\label{d.c} Fix two $\Delta^0_1(F)$ bijective \emph{coding functions} 
\begin{gather*}
\codesimple \colon  {}^{<\nat}\nat   \rightarrow \nat \\
\codefinF\colon  \bigcup_{l\in \nat} \Big[ \Big( \prod_{k < l} F(k)\Big) \times {}^l 2\Big]   \rightarrow \nat
\end{gather*}
where $\codefinF$ is \emph{appropriate for $F$} in the following sense: For each $l \in \nat$, $\fin{h} \in  \prod_{k < l} F(k)$ and $\fin{d} \in {}^l 2$,
%\[
$
\codefinF(\fin{h}, \fin{d}) < F(l).
$
%\]
This is possible by \eqref{e.growth}.
When $F$ is clear from the context 
write $\codefin (\fin{h}, \fin{d})$ for $\codefinF(\fin{h}, \fin{d})$. 
%when $(\fin{h}, \fin{d}) \in \bigcup_{l\in \nat} \Big[ \Big( \prod_{k < l} F(k)\Big) \times {}^l 2\Big] $, 
%and $\finv(m)$ and $\cinv(m)$ for the unique $\fin{h}$ and $\fin{d}$ such that $m= \codefinF(\fin{h}, \fin{d})$ for any $m\in \nat$.

\item\label{d.eBC} 
%Suppose $g \in \pspace{F}$ and $c \in {}^\nat 2$.
%\begin{enumerate}[label=(\roman*),ref=\ref{d.eBC}.\roman*]
%\item 
Let $\ed(g,c)$ be the function in  $\pspace{ F }$ given by
\[
%$
\ed(g,c) (n) = \codefin (g \res n, c \res  n ).
%$
\]

\item\label{d.i.family.c} Suppose $\fin{c} \in {}^{<\nat} 2$. Let 
\begin{equation*}
%$
\I{\fin{c}} = \{ n\in \nat  \setdef  n \equiv \sum_{i < \lh(\fin{c})} \fin{c}(i) \cdot 2^i \pmod{2^{\lh(\fin{c})}} \},
%$
\end{equation*}
i.e., the set of  $n$ with binary expansion of the form $i_k \hdots i_l \fin{c}(l-1) \hdots \fin{c}(0)$, where $l = \lh(\fin{c})$ (for some $i_l, \hdots i_k \in \{0,1\}$ and $k \in \nat\setminus l$).
\item\label{d.i.good} Suppose $\fin{c} \in {}^{<\nat} 2$. We say $\fin{c}$ is \emph{good} if and only if whenever
$n_0 < n_1$ are two consecutive elements of ${(\fin{c})^{-1}}[\{1\}]$ (i.e., $\fin{c}(n)=0$ when $n_0<n<n_1$) then 
\begin{equation*}
%$
n_1 \in \I{\fin{c} \res n_0+1}.
%$
\end{equation*}
We also say $c \in {}^\nat 2$ is good if and only if the same as above holds, i.e., if for every $n\in \nat$, $c \res n$ is good.

\item\label{d.i.domain} %Suppose $c \in {}^\nat 2$ and $g \in \pspace{F}$. 
Let 
%\begin{equation*}\label{d.domain}
$
\nuhof(g,c) =\{ 2 \cdot \# g\res n \setdef c(n)=1 \} \setminus \{ n \in \nat \setdef g(n) = \ed(g,c)(n) \}.
$
%\end{equation*}

%
\item\label{d.i.order} %Supposing $g \in \pspace{F}$, w
We define a strict partial order $\order[g]$ on $\nat$,
letting $n_0 \order[g] n_1$ if and only if the following hold:
\begin{enumerate}[(i)]
\item $n_0 < n_1$, 
\item for each $i\in \{0,1\}$ letting $(\fin{h_i}, \fin{d_i}) = \codefin^{-1}(g(n_i))$ it holds that $\fin{h_i}$ has length $n_i$ (which implies that also $\fin{d_i}$ has length $n_i$),
\item $\fin{h_0} \subseteq \fin{h_1}$,
\item $\fin{d_0} \subseteq \fin{d_1}$.
\end{enumerate}

\item\label{d.i.e.dot} %Suppose $c \in {}^\nat 2$ and $g \in \pspace{F}$. 
Define $\edh(g,c) \in \pspace{F}$ by
\begin{equation*}\label{d.domain}
%$
\edh(g,c) (n) =\begin{cases}

\ed(g,c)(n) &\text{if $n \notin \nuhof(g,c)$, or 
$(\exists n_0, n_1 \in \nuhof(g,c) \cap n)\;  n_0 \order[g] n_1$,}\\
&\text{or $c\res n$ is not good;}\\
g(n) &\text{otherwise.}
\end{cases}
%$
\end{equation*}

\item\label{d.E} We define our %$\Pi^0_1$ 
\emph{medf} $\mathcal E$ as follows:
\[
\mathcal E = \{ \edh(h,d) \setdef h \in \pspace{F}, d \in {}^\nat 2%, \text{ and $c$ is good.}  
\}.
\]

%\item\label{d.i.e.dot} Suppose $c \in {}^\nat 2$ and $g \in \pspace{F}$. 
%Define $\edh(g,c) \in \pspace{F}$ by
%\begin{equation*}\label{d.domain}
%%$
%\edh(g,c) (n) =\begin{cases}
%
%\ed(g,c)(n) &\text{if $n \notin \nuhof(g,c)$, or}\\ 
%&\text{$c\res n$ is not good, or}\\
%%$(\exists n_0, n_1 \in \nuhof(g,c) \cap n)\;  n_0 \order[g] n_1$,}\\
%&(\exists n_0, n_1 \in \nuhof(g,c) \cap n)\text{ s.t.\ letting }(\fin{h}, \fin{d}) = \codefin^{-1}(g(n)),\\  
%%n_0 \order[g] n_1$,}\\
%&\lh(\fin{h}) = n_1 \wedge
%n_0 < n_1 \wedge g(n_0) = \codefin(\fin{h}\res n_0,\fin{d} \res n_0)\\
%%&\text{or $c\res n$ is not good;}\\
%g(n) &\text{otherwise.}
%\end{cases}
%%$
%\end{equation*}
\end{enumerate}
\end{dfn}

\begin{rem}\label{r.domain}
\begin{enumerate}[1.]
\item It is very easy to see that $\mathcal E$ is $\Sigma^1_1(F)$; thus if $\mathcal E$ is an \emph{medf}, $\mathcal E$ is $\Delta^1_1(F)$ because
\[
f \notin \mathcal E \iff (\exists f' \in {}^\nat\nat)\;  f' \in \mathcal E \wedge f \mathbin{=^\infty} f' \wedge f \neq f'.
\]
\item If $(g_0,c_0) \neq (g_1, c_1)$, $e(g_0,c_0)$ and $e(g_1,c_1)$ are eventually different, where for each $i\in\{0,1\}$, $(g_i,c_i) \in \pspace{F} \times {}^\nat2$.
\item For $g \in \pspace{F}$ and $I \subseteq \nat$, $I$ is a set of $\order[g]$-comparable points if and only if 
$g \res I = \ed(h,d)\res I$ for some $h \in \pspace{F}$ and $d \in {}^\nat2$. 
\item Note that $\nuhof(g,c)$ is by definition a subset of $2\nat$; this ensures that  $g$ and $c$ be recovered from $\edh(g,c)$ in a simple fashion. This is only a matter of convenience; we could delete ``$2\cdot$'' in \ref{d.tools}\eqref{d.i.domain} (the only slight change necessary would be in the proof of Claim~\ref{c.closed} below).
\item If $\fin{c_0}, \fin{c_1} \in {}^{<\nat }2$ and $\fin{c_0} \subseteq \fin{c_1}$
then $\I{\fin{c_1}} \subseteq \I{\fin{c_0}}$.
 \item If $\fin{c_0}, \fin{c_1} \in {}^{l}2$ for some $l\in \nat$ and $\fin{c_0} \neq \fin{c_1}$, clearly  
$\I{\fin{c_0}} \cap \I{\fin{c_1}}= \emptyset$.
\item 
The set 
$
%\[
\{{c^{-1}}[\{1\}] \setdef c\in {}^\nat 2, \text{ $c$ is good}\}
%\]
$
is an almost disjoint family:
For
assume $c_0 \neq c_1$.
Find $n \in \nat$ such that $c_0\res n \neq c_1\res n$ and note  that for each $i\in\{0,1\}$, ${c_i^{-1}}[\{1\}]$ is almost contained in $\I{c_i \res n}$ since $c_i$ is good.
\end{enumerate}
\end{rem}

We now prove our main lemma.

\medskip

\noindent
{\it Proof of Lemma \ref{l.growth}.}
Fix $F$ satisfying \eqref{e.growth} and let $\mathcal E$ etc.\ be as in Definition~\ref{d.tools}.
The proof is split up into several claims. We first show:

\begin{claim}\label{c.adf}
The set 
$
%\[
\{  \nuhof(g,c) \setdef g \in \pspace{F}, c\in {}^\nat 2, \text{ $c$ is good}\}
%\]
$
is an almost disjoint family.
\end{claim}
It will facilitate our argument to introduce the following notation:
\begin{dfn}
Suppose $\fin{c} \in {}^{<\nat} 2$ and $g \in \pspace{F}$. Let 
\begin{equation*}\label{d.domain}
%$
\I[g]{\fin{c}}= \{ 2 \cdot \codesimple( g\res n ) \setdef n \in \I{\fin{c}} \}.
%$
\end{equation*}
\end{dfn}
Note for $g_0, g_1 \in \pspace{F}$, and $\fin{c_0}, \fin{c_1} \in {}^{<\nat}2$ we have that  
$\I[g_0]{\fin{c_0}} \cap \I[g_1]{\fin{c_1}}$ is finite whenever $g_0 \neq g_1$;
and $\I[g_0]{\fin{c_0}} \cap \I[g_1]{\fin{c_1}} = \emptyset$
whenever $\fin{c_0}$ and $\fin{c_1}$ are incomparable w.r.t.\ $\subseteq$.
Moreover note for the proof of Claim~\ref{c.cases} below that $\I[g_0]{\fin{c_1}} \subseteq \I[g_0]{\fin{c_0}}$ if $\fin{c_0}\subseteq \fin{c_1}$.

\medskip

We now prove the claim.

\medskip

\noindent
{\it Proof of Claim \ref{c.adf}.}
%Let $(g_i, c_i) \in \pspace{F}\times {}^\nat 2$ for each $i\in \{0,1\}$.
%If $g_0 \neq g_0$ clearly
%\[
%\{ \codesimple(g_0 \res n) \setdef n\in \nat \} \cap \{ \codesimple(g_1 \res n) \setdef n\in \nat \}
%\]
%is finite and hence so is $\nuhof(g_0, c_1)\cap \nuhof(g_1, c_1)$.
%
%Otherwise, assume $g_0 = g_1$ but $c_0 \neq c_1$.
%As that for each $i\in \{0,1\}$, $\nuhof(g_i, c_i)$ is the image of ${c_i^{-1}}''\{1\}$ under the increasing enumeration of
%$\{ 2\cdot \codesimple( g_0 \res n) \setdef n\in \nat \}$, the claim follows as
%and ${c_0^{-1}}''\{1\}$ and ${c_1^{-1}}''\{1\}$ are almost disjoint.
Let $(g_i, c_i) \in \pspace{F}\times {}^\nat 2$ for each $i\in \{0,1\}$ and suppose $(g_0, c_0) \neq (g_1, c_1)$.
Find $n^*$ such that $(g_0\res n^*, c_0 \res n^*) \neq (g_1\res n^*, c_1 \res n^*)$.
As for each $i\in \{0,1\}$ we have $\nuhof(g_i, c_i) \mathbin{\subseteq^*} \I[g_i]{c_i \res n^*}$ and
$\I[g_0]{c_0 \res n^*} \cap \I[g_1]{c_1 \res n^*}$ is finite, we are done.
\hfill \qed{\tiny Claim~\eqref{c.adf}.}

\medskip

Now it is easy to show: 

\begin{claim}\label{c.edf}
The set $\mathcal E$ is an eventually different family.
\end{claim}
\noindent
{\it Proof of claim.}
Let $f_i \in \mathcal E$ for each $i\in\{0,1\}$ and assume $f_0\neq f_1$.
Find $g_i \in \pspace{F}$ and $c_i \in {}^\nat 2$ such that
$f_i = \edh(g_i,c_i)$ for each $i\in\{0,1\}$.

Clearly $\edh(g_0,c_0)$ and $\edh(g_1,c_1)$ can only agree on finitely many points outside of 
$\nuhof(g_0,c_0) \cup \nuhof(g_1,c_1)$.
By the previous claim and by symmetry, it therefore suffices to show that the set $X$ defined by
\[
X=\{n \in \nuhof(g_0,c_0) \setminus \nuhof(g_1,c_1) \setdef \edh(g_0,c_0)(n)= \ed(g_1,c_1)(n) \}
\]
is finite.
Assume that $n_0, n_1 \in X$ and $n_0 \neq n_1$; then $\edh(g_0,c_0)(n) = \ed(g_0,c_0)(n)$ whenever $ n > n_1$
by the definition of $\edh(g_0,c_0)$ and we are done.
\hfill \qed{\tiny Claim~\ref{c.edf}.}

\medskip

The next claim is the combinatorial heart of the entire construction and the basis of the following proof that $\mathcal E$ is maximal.
\begin{claim}\label{c.cases}
For any $g \in \pspace{F}$ one of the following holds:
\begin{enumerate}
\item\label{l.case1} There exists an infinite set $I$ together with functions $h \in \pspace{F}$ and $d \in {}^\nat 2$ such that
\begin{enumerate}
\item $g \res I = \ed(h,d)\res I$, and
\item $I \cap \nuhof(h,d)$ is finite.
\end{enumerate} 
\item\label{l.case2} There exists a good function $c \in {}^\nat 2$ such that no two $n_0, n_1 \in \nuhof(g,c)$ are comparable with respect to $\order[g]$. 
\end{enumerate}

\end{claim}

We postpone the proof of the claim and first show assuming this claim that $\mathcal E$ is maximal.

%\medskip

%We give a second proof which perhaps highlights some otherwise unnoticed aspects of the situation.

%\noindent
%{\it Second proof of claim.}
%\todo{Insert proof}
%\hfill \qed{\tiny Claim~\eqref{c.cases}.}

\begin{claim}\label{c.max}
Assuming Claim~\ref{c.cases}, the eventually different family $\mathcal E$ is maximal.
\end{claim}
{\it Proof of Claim~\ref{c.max}.}
Let $g \in \pspace{F}$ be given. 
If Case~\ref{l.case1} in Claim~\ref{c.cases} holds
find an infinite set $I$, $h \in \pspace{F}$ and $d \in {}^\nat 2$ such that
\[
g \res I = \ed(h,d) \res I
\]
and $I \cap \nuhof(h,d)$ is finite.
As $\edh(h,d) \in \mathcal E$ and $\edh(h,d)$ agrees with $\ed(h,d)$ and thus with $g$ on all but finitely many points in $I$, we are done.

If on the other hand Case~\ref{l.case2} in Claim~\ref{c.cases} holds, we may find a good function $c \in {}^\nat 2$ such 
that $g$ agrees with $\edh(g,c)$ on an infinite set, and we are also done, proving maximality.
\hfill \qed{\tiny Claim~\ref{c.max}.}

\medskip

Now it is high time we prove Claim~\ref{c.cases}.

\medskip

\noindent
{\it Proof of Claim~\ref{c.cases}.}
%{\it First proof of claim.}
Write
\[
C=\{ \fin{c} \in {}^{<\nat}2\setdef \lh(\fin{c})=0 \vee \fin{c}(\lh(\fin{c})-1)=1 \}
\]
i.e., let $C$ denote set of finite sequences from $\{0,1\}$ which end in $1$ together with the empty sequence. 
Let $g \in \pspace{F}$ be given. 

Assume first that
\begin{multline}\label{e.case1}
(\exists \fin{c} \in C) (\forall n_0 \in \I[g]{\fin{c}}) 
\text{ letting $\fin{c_1} = \fin{c}\conc 0^{n_0 - \lh(\fin{c})}\conc1$ we have }\\
 \big[ g(n_0) \neq \codefin( g \res n_0, \fin{c_1}\res n_0 )\; 
 \wedge  
(\exists n_1 \in \I[g]{\fin{c_1}}) \; ( n_0 \order[g] n_1 ) \big]
\end{multline}
Fix $\fin{c}$ witnessing the existential quantifier above and let $n_0 = \min \I[g]{\fin{c}}$ noting that 
%$
\begin{equation}\label{e.init-diff}
g(n_0) \neq \codefin( g \res n_0, \fin{c_1} \res n_0 )
\end{equation}
%$ 
where $\fin{c_1} = \fin{c}\conc 0^{n_0 - \lh(\fin{c})}\conc 1$.

By \eqref{e.case1} and as $\I[g]{\fin{d}} \subseteq \I[g]{\fin{c}}$ for every $\fin{d}\in C$ which extends $\fin{c}$, we can for each $k \in \nat$ recursively find $n_{k+1}$ so that  
$n_{k+1} \in \I[g]{\fin{c}}$ and $n_k \order[g] n_{k+1}$.
Thus we can find $h \in \pspace{F}$ and $d \in {}^\nat 2$ such that
$g\res I = \ed(h,d)\res I$, where
$I = \{n_k \setdef k \in \nat\}$.

Moreover $I \cap \nuhof(h,d)$ is finite: 
Since $g(n_0) \neq \codefin( g \res n_0, \fin{c_1}\res n_0 )$, 
$(h\res n_0, d \res n_0 ) \neq  ( g \res n_0, \fin{c_1} \res n_0)$ by \eqref{e.init-diff}, and so  $\I[h]{d \res n_0}\cap \I[g]{\fin{c_1}\res n_0} = \emptyset$
while $\nuhof(h,d) \subseteq^* \I[h]{d \res n_0}$ and $I \subseteq^* \I[g]{\fin{c_1}\res n_0}$ (the latter holds since $n_k \geq n_0$ for each $k \in \nat$).
Thus we have that Case~\ref{l.case1} of the claim holds.

\medskip

Now assume to the contrary that \eqref{e.case1} fails, i.e., it holds that
\begin{multline}\label{e.case2}
(\forall \fin{c} \in C) (\exists n_0 \in \I[g]{\fin{c}}) 
\text{ s.t.\ letting $\fin{c_1} = \fin{c}\conc 0^{n_0 - \lh(\fin{c})}\conc 1$ }\\
 \big[ g(n_0) = \codefin( g \res n_0, \fin{c_1} )\; 
 \vee  
(\forall n_1 \in \I[g]{\fin{c_1}}) \; \neg( n_0 \order[g] n_1 ) \big]
\end{multline}
Let $\fin{c_0}=\emptyset$ and recursively chose $n_k$ and $\fin{c_{k+1}}$ for each $k \in \nat$ such that
$\fin{c_{k+1}} = \fin{c_k}\conc 0^{n_k - \lh(\fin{c})}\conc 1$
and
%$g(n_k) = \codefin( g \res n_k, \fin{c_{k+1}}\res n_k )$ or 
\[
g(n_k) = \codefin( g \res n_k, \fin{c_{k+1}}\res n_k ) \vee (\forall n_1 \in \I[g]{\fin{c_{k+1}}}) \; \neg( n_0 \order[g] n_1 )
\]
holds.
Letting $c = \bigcup_{k\in\nat} \fin{c_k}$, we have that
$c \in {}^\nat 2$ is good, and no $n_0, n_1 \in \nuhof(g,c)$ are comparable w.r.t.\ $\order[g]$.
Thus, Case~\ref{l.case2} of the claim holds.
\hfill \qed{\tiny Claim~\eqref{c.cases}.}

\medskip

Finally, we have:
\begin{claim}\label{c.closed}
The medf $\mathcal E$ is $\Pi^0_1(F)$.
\end{claim}

This is fairly obvious. To be able to formulate a concise proof we extend Definitions \ref{d.tools}\eqref{d.i.domain} and \ref{d.tools}\eqref{d.i.order} in a straightforward manner:
\begin{dfn} Suppose for some $l \in \nat$, $\fin{c} \in {}^{l} 2$ and $\fin{g} \in \prod_{n < l} F(n)$. 
Define
\begin{multline*}\label{d.domain}
%$
\nuhof(\fin{g},\fin{c}) =\{ 2 \cdot \# \fin g\res n \setdef n < \lh(\fin{c}) \text{ and } \fin c(n)=1 \} \setminus \\
\{ n < \lh(\fin{c})  \setdef \fin g(n) = \codefin(\fin{g}\res n,\fin{c} \res n) \}.
%$
\end{multline*}
Moreover let $n_0 \order[\fin g] n_1$ if and only if the following hold:
\begin{enumerate}[(i)]
\item $n_0 < n_1 < \lh(\fin g)$, 
\item for each $i\in \{0,1\}$ letting $(\fin{h_i}, \fin{d_i}) = \codefin^{-1}(\fin g(n_i))$ it holds that $\fin{h_i}$ has length $n_i$ (which implies that also $\fin{d_i}$ has length $n_i$),
\item $\fin{h_0} \subseteq \fin{h_1}$,
\item $\fin{d_0} \subseteq \fin{d_1}$.
\end{enumerate}
\end{dfn}

\noindent
{\it Proof of Claim~\ref{c.closed}.}
Clearly, $\mathcal E = [T]$ where $T$ is the tree consisting of those 
\[
\fin{f} \in \bigcup_{l\in \nat} \prod_{n < l} F(n)
\]
such that for any odd $m < \lh(\fin{f})$ 
letting
\[
(\fin{g},\fin{c}) = \codefin^{-1} (\fin{f}(m))
\]
we have $\lh(\fin{g})=\lh(\fin{c})=m$, and for every $n < m$,
\begin{enumerate}[label=(\Roman*),ref=\Roman*]
\item\label{i.in} $\fin{f}(n) =\fin{g}(n)$ if all of the following holds:
\begin{enumerate}

\item\label{i.closed.}\label{i.closed.first} $\fin{c}$ is good,
\item\label{i.closed.} no two $n_0, n_1 \in \nuhof(\fin{g},\fin{c})\cap n$ are comparable w.r.t.\ $\order[\fin{g}]$,
\item\label{i.closed.last} $n \in \nuhof(\fin{g},\fin{c})$;
\end{enumerate}

\item\label{i.out} if any of \eqref{i.closed.first}--\eqref{i.closed.last} above fails, 
$\fin{f}(n) = \codefin(\fin{g}\res n, \fin{c} \res n)$.

\end{enumerate}
Lastly, clearly $T$ is $\Delta^0_1(F)$.
\hfill \qed{\tiny Claim~\ref{c.closed} and Lemma~\ref{l.growth}}

\section{Proof of the main theorem}\label{s.main}

Before we give the proof of Theorem~\ref{t.main}, i.e., that there is a $\Pi^0_1$ \emph{medf} in $\mathcal N_F$, it will be convenient to give a yet broader definition of `maximal eventually different family':

\begin{dfn}
Any two functions $g_0, g_1$ with countable domain $X$ are called \emph{eventually different} if and only if
$\{x \in X \setdef g_0(x) = g_1(x) \}$ is finite.

Suppose $E \subseteq \omega$ and $F \colon E \to \omega +1$.
A set $\mathcal E$ is an \emph{eventually different family in $\prod_{n\in E} F(n)$} if and only if $\mathcal E \subseteq \mathcal \prod_{n\in E} F(n)$ and  any two distinct $g_0,g_1\in\mathcal E$ are eventually different; such a family is called \emph{maximal} (or short: a \emph{medf}) if and only if it is maximal among such families under inclusion. 
\end{dfn}

We now have the prerequisites to give a transparent proof of our main result.
This proof has a precursor in \cite{medf-closed} where we also enlarged a \emph{medf} defined on a factor space to a \emph{medf} in the entire (product) space.
\begin{proof}[Proof of Theorem~\ref{t.main}]
The proof is slightly easier if we assume that $\codesimple$ was chosen so that 
$\codesimple \emptyset  =0$, so let us make this assumption from now on.
As $\lim_{n\to\infty}F(n) = \omega$ we may find a sequence $\langle e_m \setdef m \in \nat \rangle$  which is $\Delta^0_1(F)$ such that $e_0=0$ and 
for each $m \in \nat$ we have
\begin{equation}\label{e.growth.l}
 \sum_{l\leq m} \Big( \Big\lvert  \prod_{k < l}  F(e_k) \Big\rvert \cdot 2^l \Big)  \leq F(e_m).
\end{equation}
Let $E = \{ e_m \setdef m\in \nat\}$ and let $e \colon \nat \to \nat$ denote the map $m \mapsto e_m$.

As by \eqref{e.growth.l}, $F\circ e$ satisfies the growth condition in Theorem~\ref{l.growth} the proof of said theorem gives us a $\Pi^0_1(F)$ \emph{medf} $\mathcal E_0$
in the space $\prod_{n\in E} F(n)$.
In fact, the proof gives us a $\Delta^0_1(F)$ tree $T$ such that for any $f \in \pspace{ F }$,
\begin{equation*}
f \res E \in \mathcal E_0 \iff (\forall n \in \nat) \; f \circ e \res n \in T.
\end{equation*}

For $f \in \mathcal E_0$ 
define $\edg(f) \in \pspace{F}$ as follows:
\[
\edg(f)(n) = \begin{cases}
%0  & \text{if $n\notin E$ and there is no $m$ such that}\\
%& m < n \text{ and } \codesimple \big((f \res m) \circ e\big)  < F(n),\\

\codesimple \big((f \res m) \circ e\big)  & \text{for $n\notin E$, where $m$ is maximal such that}\\
& m \leq n \text{ and } \codesimple \big((f \res m) \circ e\big)  < F(n),\\

f(n) & \text{for $n \in E$.}
\end{cases}
\]
This is well defined as $\codesimple \emptyset  < F(n)$ for all $n\in\nat$.
(It may be interesting to note that one can delete the requirement $m \leq n$ above in case $F[\nat\setminus E]\subseteq  \nat$ for the purposes of the present proof.)
%; note also the general case will be needed below in the proof of Fact~\ref{f.compact}.)

We show that $\mathcal E= \{\edg(f) \setdef f \in \mathcal E_0 \}$ is a $\Pi^0_1(F)$ \emph{medf}. 
It is maximal as 
\[
\{g \res E \setdef g \in \mathcal E \} = \mathcal E_0,
\]
and $\mathcal E_0$ is maximal in 
$\prod_{n\in E} F(n)$: 
Whenever $h \in \pspace{ F }$ there is $f \in \mathcal E_0$ such that $h\res E$ and $f$ agree on infinitely many points from $E$, so $\edg(f) \mathbin{=^\infty} h$.

\medskip

Clearly, $\mathcal E$ is also an eventually different family, as $\mathcal E_0$ is: For two distinct functions $f_0$ and $f_1$ from $\mathcal E_0$,
find $m_0 \in E$ is such for all $m \in E\setminus m_0$, $f_0(m) \neq f_1(m)$.
Further, find $n_0 > m_0$ such that for all $n \geq n_0$ 
\[
\codesimple\big((f  \res m_0 +1)\circ e \big) < F(n). 
\]
Then for all $n \in \nat \setminus n_0$  we have 
$\edg(f_0)(n) \neq \edg(f_1)(n)$.

\medskip

We show $\mathcal E$ is $\Pi^0_1(F)$.
Obviously $g \in \mathcal E$ if and only if 
for every $n,m \in \nat$ 
\begin{itemize}
\item $(g \res n ) \circ e  \in T$, and
\item whenever $n \notin E$,  and $m$ is maximal such that $m \leq n$ and $\codesimple \big((g \res m) \circ e\big)  < F(n)$,
we have $g(n) = \codesimple \big((g \res m) \circ e\big)$.
\end{itemize}
Clearly we can compute $(g \res n ) \circ e$ from $g \res n$ (relative to $F$).
Thus all the requirements in the above definition of $\mathcal E$ past the universal quantifier over $n$ and $m$ are $\Delta^0_1(F)$.
\end{proof}

\section{Proof of the corollary}\label{s.cor}

Corollary~\ref{t.general} now immediately follows using the following rather trivial facts. We include proofs for the convenience of the reader. For the remainder, fix $F \colon \nat \to (\omega + 1 )\setminus \{0\}$.
\begin{fct}\label{f.compact}
There is a compact $\Pi^0_1(F)$ \emph{medf} on $\pspace{ F }$ if and only if $F(n) < \omega$ for infinitely many $n\in \nat$.
\end{fct}
\begin{proof}
Let 
$D = \{ n \in \nat \setdef F(n) < \omega \}$.
If $D$ is finite, clearly there cannot be a compact \emph{medf}, as for every compact $\mathcal E \subseteq \pspace{ F }$ there is $ f \in \pspace{ F }$ such that $f$ eventually dominates every $g \in \mathcal E$, i.e., $\{n \in \nat \setdef f(n) \leq g(n) \}$ is finite.

Conversely, suppose $D$ is infinite and show there is a compact $\Pi^0_1(F)$ \emph{medf}.
Clearly we can assume $\lim_{n\to \infty} F(n) = \omega$ as otherwise there is a recursive finite \emph{medf} consisting of constant functions.

The proof is almost the same as that of Theorem~\ref{t.main}, so we only point out the necessary changes.
When defining $\langle e_n \setdef n\in \nat\rangle$ as in the proof of Theorem~\ref{t.main}, let $e_0$ be the least element of $D$ and when choosing $e_m$ for $m > 0$ demand in addition that $e_m\in D$ as well.

When defining $\edg(f)$ from $f \in \mathcal E_0$, demand that $\edg(f)(n) = 0$ for $n < e_0$.
The $\Pi^1_0$ condition for membership in $\mathcal E$ is easily adapted from the one in the proof of Theorem~\ref{t.main}.
The rest
of the proof can be followed verbatim; we produce a compact \emph{medf} $\mathcal E$ 
as for every $g \in \mathcal E$ and every $n\in \nat$ we have $g(n) < F(n) < \omega$ if $n \in E$ and
for $n\notin E$ we have 
\[
g(n) \in \{ \codesimple(\fin{f}\circ e) \setdef (\exists m \leq n) \; \fin{f} \in \prod_{k \in E \cap m} F(k) \}
\]
where the right-hand side is a finite set.
\end{proof}

\medskip

Finally we have:
\begin{fct}
If $\lim_{n\to\omega} F(n) = \omega$ every \emph{medf} is uncountable;
otherwise, every \emph{medf} is  finite.
\end{fct}
\begin{proof}
If $\liminf_{n\to\omega} F(n) < \omega$, every eventually different family is finite: towards a contradiction
find $m^* \in \nat$ such that
$\{ n\in \nat \setdef F(n) < m^*\}$ is infinite. By the pigeonhole principle, there is no eventually different family of size $m^*$.

If on the other hand $\lim_{n\to\omega} F(n) = \omega$, a simple diagonalization argument shows that there is no countable \emph{medf}.
\end{proof}

\section{Questions}\label{s.q}

\begin{enumerate}[1.]
\item Is it the case that for some $F\colon\nat \to \nat$ there is a compact $\Pi^0_1(F)$ maximal cofinitary group in $\pspace{ F }$?

\item For which $F$ is the answer to the previous question `yes' (if any)?
 It is easy to see that it is necessary that $F(n) > n$ for all  but finitely many $n$.

\item Is there a natural, minimal fragment of second order arithmetic which proves there is a $\Pi^0_1$ eventually different family? 
\item For any set $X$ let $X^{[\infty]}$ denote the set of infinite subsets of $X$. 
Given any $F\colon \nat \to \{ \nat\}\cup\nat$ and a \emph{medf} $\mathcal E$ on $\pspace{ F }$ consider the co-ideal
\[
\mathcal C_{\mathcal E} = \{ X  \in \powerset(\nat) \setdef \{ g \res X \setdef g \in \mathcal E \}\text{ is a \emph{medf} in $\prod_{n\in E}F(n)$} \}.
\]
Is there  a closed \emph{medf} $\mathcal E$ in $\mathcal N$ or $\pspace{ F }$ (under some assumption on $F$) such that
$\mathcal C_{\mathcal E} = \nat^{[\infty]}$?
\end{enumerate}

 \bibliography{compactness-medf}{}
\bibliographystyle{amsplain}
 
\end{document}